\documentclass[11pt]{amsart}

\usepackage{amsmath, amsthm, amssymb}
\usepackage{color}
\usepackage{geometry}

\renewcommand{\d}{\partial}
\newcommand{\ddbar}{\sqrt{-1}\d\overline{\d}}

\newcommand{\Ric}{\mathrm{Ric}}

\newcommand{\tr}[2]{\mathrm{tr}_{#1}{#2}}

\newcommand*\Laplace{\Delta}

\numberwithin{equation}{section}

\def\CC{\mathbb{C}}
\def\RR{\mathbb{R}}
\def\P{\mathbb{P}}

\newlength{\ten}
\newlength{\four}
\newtheorem{thm}{Theorem}[section]
\newtheorem{prop}[thm]{Proposition}
\newtheorem{lem}[thm]{Lemma}
\newtheorem{cor}[thm]{Corollary}

\theoremstyle{definition}

\newtheorem{remark}[thm]{Remark}

\renewcommand{\leq}{\leqslant}
\renewcommand{\geq}{\geqslant}
\renewcommand{\epsilon}{\varepsilon}

\begin{document}
\title{A Scalar Curvature Bound Along the Conical K\"ahler-Ricci Flow}
\author{Gregory Edwards}

\begin{abstract}
Starting with a model conical K\"ahler metric, we prove a uniform scalar curvature bound for solutions to the conical K\"ahler-Ricci flow assuming a semi-ampleness type condition on the twisted canonical bundle. In the proof, we also establish uniform estimates for the potentials and their time derivatives.
\end{abstract}

\maketitle

\section{Introduction}

The K\"ahler-Ricci flow has received a considerable amount of attention over the past decades following Cao's parabolic proof of Yau's Theorem \cite{cao,yau} and has since become an essential tool in geometric analysis. The flow has been investigated extensively and understanding its behavior has become a central object of study in K\"ahler geometry (see for instance \cite{PSSW09,PS06,SesTi,SoTi09,SoTi07,SoTi12,TiZha13,TiZhu07} and the references therein).

Recently, the problem of finding K\"ahler-Einstein metrics with cone singularities along a divisor has become an enticing problem for geometers. Conical K\"ahler-Einstein metrics have been around for some time, yet recently conical K\"ahler metrics have generated renewed interest \cite{berman,brendle,CGP,CDS1,CDS2,CDS3,Do11,GP,JMR,LiSun,Ru14,SoWa12}.

The conical K\"ahler-Ricci flow was introduced as a parabolic flow which preserves conical singularities. Chen-Wang showed short time existence of the flow \cite{ChenWang1} and proved that the solution remains conical as long as the flow exists. Later Shen \cite{Shen} showed the flow exists on a maximal time interval which is determined by cohomology. Liu-Zhang studied the conical K\"ahler-Ricci flow on Fano manifolds where they proved that the flow converges to a conical K\"ahler-Einstein metric, if one exists \cite{LiZh}. The conical flow has also been extensively studied on Riemann surfaces \cite{MRS,PSSWa15,PSSWa14}.

Of particular importance to the study of K\"ahler-Ricci flow is the connection between the long time behavior of the flow and the canonical bundle of the underlying manifold. Tsuji \cite{Ts} and Tian-Zhang \cite{TiZh} showed that on a minimal model of general type the K\"ahler-Ricci flow converges to a unique singular K\"ahler-Einstein metric outside of a codimension one subvariety. Zhang then established that in such a setting the scalar curvature actually remains uniformly bounded along the flow \cite{Zh}. This was generalized in the work of Song-Tian \cite{SoTi11} who showed that the scalar curvature remains uniformly bounded on any minimal model with semi-ample canonical bundle.

Most notably the K\"ahler-Ricci flow has been used to carry out an analytic version of the Minimal Model Program (MMP) in algebraic geometry. Roughly speaking, one starts with a compact K\"ahler manifold and through a series of algebraic surgeries produces a manifold in the same birational class which is either a Mori fiber space or a \emph{minimal model}, meaning that the canonical bundle is nef. Remarkably, in \cite{SoTi09,SoTi07,SoTi12} Song-Tian proved that when the K\"ahler-Ricci flow approaches a finite time singularity, one can perform a canonical algebraic surgery and continue the flow in a weak sense on the resulting manifold. The process follows the steps of the MMP and can be continued at each finite time singularity until either the volume of the manifold collapses, and so the final manifold is the required Mori fiber space, or the Kahler-Ricci flow exists for all time, in which case the final manifold is the minimal model in the birational class (see also \cite{CoTo15,FoZh12,Gi14,GrToZh13,SoWe11,SoWe11a,ToWeYa14}). In the later case, it is expected that the flow converges at infinity to a K\"ahler-Einstein metric on the canonical model.

It is a natural question to ask if similar results can be established between the conical K\"ahler-Ricci flow and the Minimal Model Program on logarithmic (log) pairs. Specifically, on a compact K\"ahler manifold $M$ with cone angle $2\pi\beta$ along a given smooth divisor $D$, one can consider $(M,(1-\beta)D)$ as a log pair and ask if the conical K\"ahler-Ricci flow carries out the log Minimal Model Program in a canonical fashion.

In the conical setting, some preliminary results have already been shown to relate the long time behaviour of the conical K\"ahler-Ricci flow to the so-called \textit{twisted canonical bundle} $K_M + (1-\beta)D$. Indeed, Shen also showed that if the twisted canonical bundle is big and nef, then the flow converges to a unique singular conical K\"ahler-Einstein metric outside of a subvariety of codimension at least one \cite{Shen}. As an additional result in this direction, we prove a scalar curvature bound for log minimal pairs which satisfy a semi-ampleness condition on the twisted canonical bundle. This generalizes the scalar curvature bound of Song-Tian to the conical setting.

When the twisted canonical bundle is big and nef, Shen also established bounds on the potential and its time derivative which deteriorate as one approaches the null locus of twisted canonical bundle \cite{Shen}. In the context of our paper we obtain uniform bounds on the potential and its time derivative globally. When the twisted canonical bundle is big and nef this improves the estimates of Shen (see Corollary \ref{BCHM}).

Let $M$ be a compact K\"ahler manifold. We say $\omega^*$ is a conical K\"ahler metric with cone angle $2\pi\beta$ ($0<\beta<1$) along a smooth divisor $D$ if it is a smooth K\"ahler metric on $M \setminus D$ and is asymptotically equivalent along $D$ to the local model metric on $\CC^n$,
\begin{equation}\label{eq:asymptotic}
	\sqrt{-1}\frac{d z_1 \wedge d \bar z_1}{|z_1|^{2(1-\beta)}} + \sum_{j=2}^{n} { \sqrt{-1} dz_j \wedge d \bar z_j}
\end{equation}
where $(z_1,...,z_n)$ are local holomorphic coordinates such that $D=\{z_1=0\}$.

The definition of such a cone metric can be generalized to a metric having singularities along simple normal crossing divisors $D_1,...,D_\ell$ with cone angles $2\pi\beta_k$ ($0<\beta_k<1$) along $D_k$ by requiring that if $p \in \bigcap_{j=0}^r D_{k_j}$ (with $p \notin D_k$ for any $k\notin \{k_1,...,k_r\}$), then $\omega^*$ is asymptotically equivalent to

\[ \sum_{j=1}^r \sqrt{-1} \frac{d z_j \wedge d \bar z_j}{|z_j|^{2(1-\beta_{k_j})}} + \sum_{j=r+1}^{n} { \sqrt{-1} dz_j \wedge d \bar z_j} \]
in a holomorphic coordinate chart centered at $p$ such that locally $D_{k_j} = \{z_j = 0\}$ for $j=1,...,r$.

For simplicity we restrict to the case where $\ell=1$; that is, where there is only a single smooth divisor $D$. We merely remark that straightforward alterations would extend our results to the case of simple normal crossing divisors.

For any conical metric $\omega^*$ we define the Ricci current of $\omega^*$ to be
\[ \Ric(\omega^*) = - \ddbar \log \det(g^*) \]
where locally $\omega^* = \sqrt{-1} g^*_{i \bar j} dz_i \wedge d \bar z_j$. 

Note that $\log \frac{\det(g^*)}{\det(g_0)} \in L^1(M)$ for any smooth reference metric $g_0$, so $\ddbar\log\det(g^*)$ is well-defined as a current on $M$. Outside the divisor $D$, where $\omega^*$ is smooth, this reduces to the classical Ricci form.

Let $\omega_0$ be a K\"ahler metric on $M$. If $h$ is a Hermitian metric on the line bundle $\mathcal{O}_M(D)$ and $S$ is a section vanishing along $D$, it is well known that for sufficiently small positive constant $k$,
\[ \omega^* = \omega_0 + k \ddbar \|S\|_h^{2\beta} \]
is a conical K\"ahler metric with cone angle $2\pi\beta$ along $D$.

We say $\omega$ is a solution to the \textit{normalized conical K\"ahler-Ricci flow} starting with a model conical K\"ahler metric $\omega^*$ if it satisfies, in the sense of currents,
\begin{equation}\label{eq:CKRF}
\begin{cases}
\frac \d{\d t} \omega = -\Ric(\omega) - \omega + 2\pi(1-\beta)[D] \\ 
\omega(\cdot,0) = \omega^* = \omega_0 + k\ddbar\|S\|_h^{2\beta}
\end{cases}
\end{equation}
where $[D]$ is the current of integration along $D$. From \cite{ChenWang1} (see also Theorem 2.1 in \cite{Shen} using results of \cite{CGP,GP}) we know that the flow remains conical as long as it exists and so this equation is well defined.

Furthermore, by Shen \cite{Shen} a unique solution to ~\eqref{eq:CKRF} exists for all time if and only if the twisted canonical bundle, $K_M + (1-\beta)D$, is nef, or equivalently, $(M,(1-\beta) D)$ is a log minimal model.

Presently we obtain the following theorem generalizing the scalar curvature bound of Song-Tian \cite{SoTi11} to the conical setting.

\begin{thm}\label{main}
Let $M$ be a compact K\"ahler manifold of complex dimension $n$ and fix cone angle $0<\beta<1$. Assume there exists a morphism, 
\[ \pi: M \rightarrow Y \subseteq Z, \]
from $M$ to a smooth compact K\"ahler manifold $Z$ with image an irreducible normal subvariety $Y$, and a divisor $D'\in \mathrm{Div}(Y)$ such that $D=\pi^*D'$ defines a smooth divisor on $M$.

In addition, assume there exists a K\"ahler class $[\omega_Z]\in H^{1,1}(Z; \RR)$ such that
\[ \pi^*\omega_Z \in [K_M + (1-\beta)D]. \]
Then for any K\"ahler metric $\omega_0$ there exists constants $C,k_0>0$ such that for all $0<k<k_0$, the solution to ~\eqref{eq:CKRF} satisfies
\[ |R(t)| \leq C \text{,  on  } M \setminus D \text{ for all } t\in[0,\infty) \]
where $R(t) = \tr{\omega}{\Ric(\omega)}$ is the scalar curvature of the evolving metric.
\end{thm}

\begin{remark}
Although we do not state its precise form, the constant $k_0$ can be made completely explicit and can be determined by the interested reader.
\end{remark}

\begin{remark}
We have two situations in mind for the setup of the theorem. The first is if $Y$ itself is smooth, then we take $Z=Y$. The second is when $Z$ is either $\P^N$ or a product of projective spaces and $Y \subseteq Z$ is a normal subvariety.
\end{remark}

Our method of proof follows that of Song-Tian \cite{SoTi11}, however because the conical metrics are not smooth our result does not follow immediately from their approach. Instead we adopt the technique of Liu-Zhang \cite{LiZh}, which uses the smooth approximation metrics developed in Campana-Guenancia-Paun \cite{CGP} and Guenancia-Paun \cite{GP} to obtain a family of smooth solutions which approximate the solutions to the singular equation. We prove uniform estimates on the scalar curvature of these approximation solutions, and then, by the argument of Shen \cite{Shen}, these approximations converge to a solution of the conical K\"ahler-Ricci flow from which we obtain our uniform estimate along the flow.

Recall that we say $M$ is a minimal model if the canonical divisor $K_M$ is nef, and we say a line bundle is semi-ample if any sufficiently large multiple is globally generated. As an application of the previous theorem we obtain the following useful corollary.

\begin{cor}
Let $M$ be a projective minimal model of Kodaira dimension $\kappa = Kod(M)$ and assume $K_M$ semi-ample. Let
\[ \pi: M \rightarrow M_{can} \subseteq \P^N \]
be the morphism from $M$ to its canonical model determined by the pluricanonical system $|m K_M|$ for $m$ sufficiently large. Thus $K_M = \pi^*\mathcal{O}_{\P^N}(1)$, and if $D'$ is a divisor on $M_{can}$ such that $D=\pi^*D'$ is smooth, then the scalar curvature of the solution to the normalized conical K\"ahler-Ricci flow starting with conical metric $\omega_0 + k\ddbar \|S\|_h^2$ is uniformly bounded away from $D$, for all $k$ sufficiently small.
\end{cor}

If we assume the abundance conjecture, which predicts that whenever $K_M$ is nef it must be semi-ample, the previous corollary applies to all minimal models.

As a particular case, the above corollary establishes a uniform scalar curvature bound if $D \in |\lambda K_M|$ for $\lambda$ any positive constant when $K_M$ is semi-ample. 

In the complementary setting where $M$ is Fano, the conical K\"ahler-Ricci flow has previously been studied where $D$ is an effective divisor proportional to the canonical bundle. If $D \in |-\lambda K_M|$, for $\lambda$ positive constant, this was first studied by Chen-Wang \cite{ChenWang2} who showed long time existence of the flow when $1 - (1-\beta)\lambda \leq 0$, that is, when the twisted canonical bundle is ample or trivial. This same situation was later studied by Liu-Zhang \cite{LiZh} who showed long time existence of the properly normalized flow and established a scalar curvature bound when $1 - (1-\beta)\lambda > 0$, i.e. when the twisted canonical bundle is anti-ample.

Recall that we say an $\RR$-divisor $E$ is big and nef if it is nef and satisfies
\[ [c_1(E)]^n = \int_M{c_1(E)^n} > 0 .\]
Our theorem also has applications to situations where the twisted canonical bundle is big and nef.

\begin{cor}\label{BCHM}
Let $M$ be a compact K\"ahler manifold with smooth irreducible divisor $D$ and suppose $K_M + (1-\beta)D$ is big and nef for some $0<\beta<1$. If 
\[ D \cdot (K_M + (1-\beta)D)^{n-1} = \int_D{c_1(K_M + (1-\beta)D)^{n-1}} > 0, \]
then the scalar curvature of the solution to the normalized conical K\"ahler-Ricci flow is uniformly bounded outside $D$. Moreover, the potential and its time derivative are uniformly bounded along the flow.
\end{cor}

The organization of the rest of the paper is as follows: In Section 2, we introduce a sequence of twisted K\"ahler-Ricci flows which approximate the conical K\"ahler-Ricci flow, and we prove Corollary \ref{BCHM}. In Section 3, we establish uniform bounds on the potentials of the approximation equations and their time derivatives. In Section 4, we establish uniform lower bounds on the evolving scalar curvature of the twisted K\"ahler-Ricci flows. In Section 5, we reduce the upper bounds on scalar curvature to a Laplacian estimate. In Section 6, we prove a parabolic Schwarz lemma for the twisted K\"ahler-Ricci flows. In Section 7, we obtain a uniform gradient estimate, and in Section 8 we prove the Laplacian estimate. Finally in Section 9, we discuss the convergence of the twisted K\"ahler-Ricci flows to a solution of the conical K\"ahler-Ricci flows and prove Theorem \ref{main}.

\section*{Acknowledgements}
This article would not exist were it not for the continued support and counsel of my thesis advisors Valentino Tosatti and Ben Weinkove. The author thanks them for their many helpful conversations and improvements toward the final version of this paper.

\section{Preliminaries}

Fix an initial K\"ahler metric $\omega_0$ on $M$ and let $\kappa = \dim (Y)$. Since $\omega_Z|_Y$ is a K\"ahler metric on $Y$ we can define a semi-positive (1,1)-form on $M$,
\[ \widehat\omega_\infty = \pi^* \omega_Z|_Y .\]
We use $\widehat\omega_\infty$ to define the family of smooth K\"ahler metrics,
\[\omega_t = e^{-t}\omega_0 + (1-e^{-t})\widehat{\omega}_\infty .\]

Now, fix a Hermitian metric $h'$, in the sense of \cite{EGZ09}, on $\mathcal{O}_Y(D')$ along with a section $S'$ vanishing along $D'$ normalized such that
\[ \sup_{Y} \|S'\|_{h'}^2 = 1 .\]
Since $D = \pi^* D'$ (i.e. $\mathcal{O}_M(D) = \pi^* \mathcal{O}_Y(D')$) we can define a Hermitian metric $h = \pi^* h'$ on $\mathcal{O}_M(D)$ by
\[ h(x) = \pi^* h'(x) = h' \circ \pi(x) \]
and the curvature form of $h$ is precisely
\[ R_h = - \ddbar \log h = - \ddbar \log \pi^* h' = - \pi^* \ddbar \log h' = \pi^* R_{h'} .\]
Then $R_{h'}$ is smooth in the sense of \cite{EGZ09}, and $R_{h}$ is a smooth (1,1)-form on $M$.
Similarly, $S = \pi^*S'$ defines a global section of $\mathcal{O}_M(D)$ vanishing along $D$ and satisfies $\sup_M {\|S\|_h^2} = 1$.

We then define the sequence of reference cone metrics, 
\[\overline\omega_t = \omega_t + k\ddbar\|S\|_h^{2\beta}\]
where $k$ is a small positive constant.

Note that we can always choose $k$ sufficiently small such that $\overline\omega_t$ remains positive for all $t\geq 0$. Indeed, let $C_0>0$ be a constant such that
\[ C_0 \omega_Z|_Y \geq R_{h'} \geq - C_0 \omega_Z|_Y .\]
Then, after pulling back by $\pi$, we obtain
\begin{equation}\label{eq:C0}
	C_0 \widehat\omega_\infty \geq R_h \geq -C_0 \widehat\omega_\infty.
\end{equation}
Thus we calculate
\begin{align*}
	\ddbar \|S\|_h^{2\beta} & = \beta^2 \frac{\sqrt{-1} \nabla S \wedge \overline{\nabla S}} {\|S\|_h^{2(1-\beta)}} - \beta \|S\|_h^{2\beta} R_h \\
	& \geq \beta^2 \frac{\sqrt{-1} \nabla S \wedge \overline{\nabla S}} {\|S\|_h^{2(1-\beta)}} - \beta C_0 \widehat\omega_\infty \\
	& \geq -\beta C_0 \widehat\omega_\infty
\end{align*}
where we are writing $\nabla$ for the Chern connection on $\mathcal{O}_M(D)$ associated to the Hermitian metric and $\sqrt{-1} \nabla S \wedge \overline{\nabla S}$ is the (1,1)-form given in local coordinates (with respect to some trivialization of the line bundle) by
\[ \sqrt{-1} \nabla S \wedge \overline{\nabla S} = h \nabla_i S \overline{\nabla_j S} \sqrt{-1} dz_i \wedge d \bar z_j .\]

Thus for $k$ sufficiently small,
\[ \widehat\omega_\infty + k\ddbar\|S\|_h^{2\beta} \geq (1 - k\beta C_0) \widehat\omega_\infty \geq 0 ,\]
and therefore for all $t\in[0,\infty)$,
\[ \overline\omega_t = e^{-t}(\omega_0 + k\ddbar\|S\|_h^{2\beta}) + (1-e^{-t})(\widehat\omega_\infty + k\ddbar\|S\|_h^{2\beta}) \geq C^{-1} \omega_t .\]

Now, fix a smooth volume form $\Omega$ such that
\[ \ddbar \log \Omega = \widehat\omega_\infty - (1-\beta)R_h \in - c_1(M)  ,\]
and normalized such that 
\[\int_M \Omega = \int_M \omega_0^n.\]

Then if we set $\omega = \bar{\omega}_t + \ddbar\varphi$, $\omega$ will satisfy the conical K\"ahler-Ricci flow if and only if
\begin{multline}
	\ddbar\frac \d{\d t}\varphi = \ddbar \log(\bar\omega_t + \ddbar\varphi)^n + 		(1-\beta)\ddbar\log\|S\|_h^2 \\
	+ (1-\beta)R_h - \bar{\omega}_t - \ddbar\varphi\ - \frac \d{\d t} \omega_t.
\end{multline}
Altering $\varphi$ by a function depending only on time, this is equivalent to the parabolic Monge-Amp\`ere equation
\begin{equation}\label{eq:CMA}
	\begin{cases}
		\frac \d{\d t} \varphi = \log\frac{e^{(n-\kappa)t}(\bar\omega_t + 	\ddbar\varphi)^n}{\Omega} + \log\|S\|_h^{2(1-\beta)} - \varphi - k\|S\|_h^{2\beta}\\
		\varphi(\cdot,0)=0.
	\end{cases}
\end{equation}

Conversely, if $\omega$ is a solution of conical K\"ahler-Ricci flow ~\eqref{eq:CKRF} we let $\varphi$ be a solution of the ODE
\begin{equation}\label{eq:ODE}
	\frac \d{\d t} \varphi = \log \frac{e^{(n-\kappa)t}\omega^n \|S\|_h^{2(1-\beta)}}{\Omega} - \varphi - k \|S\|_h^{2\beta} \text{, } \varphi|_{t=0}=0.
\end{equation}
Evidently,
\[ \frac \d{\d t} (e^t(\omega - \bar\omega_t - \ddbar\varphi)) = 0 \text{, } (\omega - \bar\omega_t -\ddbar\varphi)|_{t=0} = 0 ,\]
and so $\omega = \bar\omega_t + \ddbar\varphi$. Plugging back into ~\eqref{eq:ODE} we obtain ~\eqref{eq:CMA}.

By Shen \cite{Shen}, if the twisted canonical bundle is nef, a solution of ~\eqref{eq:CMA} exists and is unique for all $t\in[0,\infty)$. 

We wish to prove that the scalar curvature of this metric is uniformly bounded outside $D$. To do so we consider a smooth approximation of ~\eqref{eq:CMA}.

We adopt the method of Liu-Zhang \cite{LiZh} and first define a family of smooth metrics which closely approximate the family of cone metrics. Following Campana-Guenancia-Paun \cite{CGP} we define
\[\omega_{t,\epsilon} = \omega_t + k\ddbar\chi(\|S\|_h^2 + \epsilon^2)\]
where,
\[\chi(\epsilon^2 + t) \equiv \beta \int_0^t{\frac{(\epsilon^2 + r)^\beta - \epsilon^{2\beta}}{r}dr} .\]

Note that $\omega_{t,\epsilon}$ is a smooth positive K\"ahler metric for all $t\in[0,\infty)$ and for all $\epsilon>0$. For each $t\in[0,\infty)$, $\omega_{t,\epsilon}$ converges to $\bar\omega_t$ globally on $M$ in the sense of currents and in $C^\infty_{loc}$ on $M \setminus D$ as $\epsilon$ tends to zero. Further, $\chi(\epsilon^2 + t)$ is positive and uniformly bounded independent of $\epsilon$ provided $t$ takes values in a bounded range, and there is uniform constant $\gamma>0$ independent of $\epsilon$ such that $\omega_{0,\epsilon} \geq \gamma \omega_0$.

We now consider smooth approximation equations to ~\eqref{eq:CMA},
\begin{equation}\label{eq:smooth}
	\begin{cases}
		\frac \d{\d t} \varphi_\epsilon = \log \frac{e^{(n-\kappa)t}(\omega_{t,\epsilon} + \ddbar\varphi_\epsilon)^n}{\Omega} + \log (\|S\|_h^2 + \epsilon^2)^{(1-\beta)} -\varphi_\epsilon - k\chi\\
		\varphi_\epsilon(\cdot,0)=0.
	\end{cases}
\end{equation}
Here we are using the shorthand $\chi = \chi(\|S\|_h^2 + \epsilon^2)$. Although $\chi$ depends on $\epsilon$, we suppress this dependence and shall use this convention throughout.

Finally, letting $\omega_{\varphi_\epsilon} = \omega_{t,\epsilon} + \ddbar\varphi_\epsilon$, we note that this parabolic complex Monge-Amp\`ere equation is equivalent to the smooth generalized K\"ahler-Ricci flow defined by,
\begin{equation}\label{eq:gen}
	\begin{cases}
		\frac \d{\d t} \omega_{\varphi_\epsilon} = -\Ric(\omega_{\varphi_\epsilon}) - \omega_{\varphi_\epsilon} + (1-\beta)\ddbar\log(\|S\|^2_h + \epsilon^2) +(1-\beta)R_h\\
		\omega_{\varphi_\epsilon}(\cdot,0) = \omega_{0,\epsilon}. 
	\end{cases}
\end{equation}
Such twisted K\"ahler-Ricci flows have been studied for instance in \cite{CoSz,LiZh,Shen,SoTi12}.

Finally, we conclude this section with a proof of Corollary \ref{BCHM}.

\begin{proof}[Proof of Corollary \ref{BCHM}]
Due to the Base Point Free Theorem for $\RR$-divisors (cf. \cite{BiCaHaMc06,HaMc}) whenever the twisted canonical bundle is big and nef there is a normal projective variety $Y$ and a surjective morphism $\pi:M \rightarrow Y$ with connected fibers such that
\[ K_M + (1-\beta)D = \pi^*H \]
for some ample $\RR$-divisor $H$ on $Y$.

Note that the ample $\RR$-divisor $H$ embeds $M$ into a product of projective spaces which allows us to pullback a K\"ahler class representing $c_1(H)$. Indeed if
\[ H = \sum_{i=1}^r {a_i H_i},\]
for $a_i>0$ and $H_i$ ample divisors, then by fixing constants such that $m_i H_i$ is very ample we have the embeddings
\[ \Phi_{|m_i H_i|} : Y \rightarrow \P^{N_i}, \text{ for } i = 1,..,r. \]
Let $\omega_i$ be the Fubini-Study metric on $\P^{N_i}$ for $i=1,...,r$, then
\[ \omega_Y = \sum_{i=1}^r {\frac{a_i}{m_i} \Phi_{|m_i H_i|}^*\omega_i} \in c_1(H), \]
and so $\pi^*\omega_Y \in [K_M + (1-\beta)D]$. Now, letting $Z = \P^{N_1} \times ... \times \P^{N_r}$ and $p_i$ the projection onto the $i^{th}$ factor, we have the embedding,
\[\Phi = \Phi_{|m_1 H_1|}\times...\times\Phi_{|m_r H_r|}:Y \rightarrow Z,\]
so that
\[ \omega_Y = \Phi^*(\sum_{i=1}^r {\frac{a_i}{m_i} p_i^*\omega_i}), \]
which establishes the existence of the K\"ahler class in the theorem.

Now, because $K_M + (1-\beta)D = \pi^*H$ is big, $\pi$ must be generically finite; and since $\pi$ has connected fibers it must in fact be a birational morphism.

Let $S$ be the null locus of $\pi^*H$, that is
\[ S = \textrm{Null}(\pi^*H) = \bigcup_{V\cdot(\pi^*H)^{\textrm{dim} V}=0}{V} \]
where the union on the right hand side is over all irreducible algebraic subvarieties $V \subseteq M$.

Since $H$ is ample, the null locus of $\pi^*H$ is precisely the union of all subvarieties which are contracted by $\pi$ and thus $\pi$ is an isomorphism on $M \setminus S$.

Since 
\[ D \cdot (K_M + (1-\beta)D)^{n-1} = D \cdot (\pi^* H)^{n-1} \neq 0 \]
and $D$ is irreducible, $D$ is not contained in $S$. Therefore $\pi(D)=D'$ defines a divisor on $Y$ and $D$ is its strict transform. Thus we have $D = \pi^*D'$, and so the previous theorem applies.  \end{proof}

That the potential and its time derivative are uniformly bounded follows from the estimates in the ensuing section and the convergence results discussed in Section 9 (cf. \cite{Shen}).

\section{Uniform Bound on Potential}

Let $\varphi_\epsilon$ be a solution to ~\eqref{eq:smooth} so that $\omega_{\varphi_\epsilon} = \omega_{t,\epsilon} + \ddbar \varphi_\epsilon$ solves the generalized K\"ahler-Ricci flow equation ~\eqref{eq:gen}. In this section we show $\varphi_\epsilon$ and $\dot\varphi_\epsilon$ are uniformly bounded.

\begin{prop}
There is constant $C>0$ such that on $M\times[0,\infty)$
\[ |\varphi_\epsilon| \leq C ,\]
and
\[ \frac \d{\d t} \varphi_\epsilon \leq C .\]
\end{prop}

\begin{proof}
First, since $\widehat\omega_\infty$ is the pullback of a metric on $Y$ we have
\[ \widehat\omega_\infty^{\kappa + 1} = 0 \text{ on } \pi^{-1}(Y_{reg})\]
where $Y_{reg}$ is the Zariski open subset where $Y$ is smooth. Thus since $\widehat\omega_\infty^{\kappa + 1}$ is smooth and vanishes on an open dense subset of $M$, we have
\[ \widehat\omega_\infty^{\kappa + 1} = 0 \text{ on } M, \]
and therefore there is a constant $C>0$ such that
\[ C^{-1} e^{-nt}\Omega \leq \omega_t^n \leq C e^{-(n-\kappa)t}\Omega .\]
Moreover, since $S = \pi^*S'$ and $h = \pi^*h'$
\[ \ddbar\chi(\|S\|_h^2 + \epsilon^2) = \ddbar\chi(\pi^*(\|S'\|_{h'}^2) + \epsilon^2) = \pi^*\ddbar\chi(\|S'\|_{h'}^2 + \epsilon^2) .\]
Thus,
\[ \widehat\omega_\infty + k\ddbar\chi(\|S\|_h^2 + \epsilon^2) = \pi^*\big(\omega_Z|_Y + k\ddbar\chi(\|S'\|_{h'}^2 + \epsilon^2) \big). \]
Since the form on the left hand side is smooth, by the same argument as above:
\[ (\widehat\omega_\infty + k\ddbar\chi)^{\kappa+1} = 0 .\]
Hence we obtain
\begin{equation}\label{eq:kod}
	\omega_{t,\epsilon}^n \leq C e^{-(n-\kappa)t} \omega_{0,\epsilon}^{n-\kappa} \wedge \omega_{t,\epsilon}^\kappa.
\end{equation}
Next, we have from the calculation in \cite{CGP},
\begin{equation}\label{eq:ddbarchi}
	\ddbar \chi = \beta^2 \frac{\sqrt{-1} \nabla S \wedge \overline{\nabla S}}{(\|S\|_h^2 + \epsilon^2)^{1-\beta}} - \beta ((\|S\|_h^2 + \epsilon^2)^\beta - \epsilon^{2\beta}) R_h,
\end{equation}
where $\sqrt{-1} \nabla S \wedge \overline{\nabla S}$ is defined as above. As in \cite{Shen}, this yields
\begin{equation}\label{eq:shen}
	C^{-1} \frac{\Omega}{(\|S\|_h^2 + \epsilon^2)^{1-\beta}} \leq \omega_{0,\epsilon}^n \leq C \frac{\Omega}{(\|S\|_h^2 + \epsilon^2)^{1-\beta}}
\end{equation}
for uniform constant independent of $\epsilon$.

Next, we claim that for all $k$ sufficiently small there is a uniform constant $C>0$ such that
\[ C\omega_{0,\epsilon} \geq \omega_{t,\epsilon}.\]
Indeed, let $C_0$ be defined as in ~\eqref{eq:C0} and let $C_1>0$ be a constant such that
\[ C_1 \omega_0 > \widehat\omega_\infty .\]
Without loss of generality, we may assume $C_1>1$.

We remark that by the Minkowski inequality we have
\[ 0 \leq (\|S\|_h^2 + \epsilon^2)^\beta - \epsilon^{2\beta} \leq \|S\|_h^{2\beta} \leq 1, \]
and so using ~\eqref{eq:C0},
\[ -C_0 \widehat\omega_\infty \leq ((\|S\|_h^2 + \epsilon^2)^\beta - \epsilon^{2\beta}) R_h \leq C_0 \widehat\omega_\infty. \]
Then using ~\eqref{eq:ddbarchi} and choosing $k \leq \frac{1}{3 \beta C_0 C_1}$ we have,
\begin{align*}
	\widehat\omega_\infty + k\ddbar\chi 
	& = \widehat\omega_\infty - k\beta((\|S\|_h^2 + \epsilon^2)^\beta - \epsilon^{2\beta})R_h + k\beta^2 \frac{\sqrt{-1} \nabla S \wedge \overline{\nabla S}}{(\|S\|_h^2 + \epsilon^2)^{1-\beta}} \\
	& \leq (1 + k\beta C_0)\widehat\omega_\infty + k\beta^2 \frac{\sqrt{-1} \nabla S \wedge \overline{\nabla S}}{(\|S\|_h^2 + \epsilon^2)^{1-\beta}} \\
	& \leq (2 - 2 k\beta C_0 C_1)\widehat\omega_\infty + 2 C_1 k\beta^2 \frac{\sqrt{-1} \nabla S \wedge \overline{\nabla S}}{(\|S\|_h^2 + \epsilon^2)^{1-\beta}} \\
	& \leq 2C_1 \Big( \omega_0 - 2 k\beta((\|S\|_h^2 + \epsilon^2)^\beta - \epsilon^{2\beta})R_h + k\beta^2 \frac{\sqrt{-1} \nabla S \wedge \overline{\nabla S}}{(\|S\|_h^2 + \epsilon^2)^{1-\beta}} \Big) \\
	& = 2 C_1(\omega_0 + k\ddbar\chi).
\end{align*}
So we conclude, for uniform $C>0$ independent of $\epsilon$,
\[ \omega_{t,\epsilon} = e^{-t}(\omega_0 + k\ddbar\chi) + (1-e^{-t})(\widehat\omega_\infty + k\ddbar\chi) \leq C \omega_{0,\epsilon} .\]
Thus using ~\eqref{eq:kod} and ~\eqref{eq:shen} there is another uniform $C>0$ such that
\[ \omega_{t,\epsilon}^n \leq C e^{-(n-\kappa)t}\frac{\Omega}{(\|S\|^2_h + \epsilon^2)^{1-\beta}} .\]

Now, if $\varphi_\epsilon$ achieves a maximum at $(x_0,t_0)$ with $t_0>0$, then we have
\[ \varphi_\epsilon(x_0,t_0) \leq (\varphi_\epsilon + \dot\varphi_\epsilon)(x_0,t_0) \leq \log \frac{e^{(n-\kappa)t}\omega_{t,\epsilon}^n(\|S\|_h^2 + \epsilon^2)^{1-\beta}}{\Omega} - k\chi \leq C.\]
Hence $\varphi_\epsilon$ is bounded above.

We now prove the upper bound on $\dot\varphi_\epsilon$.
First, we have the evolution equations
\[ (\frac \d{\d t} - \Laplace)(\varphi_\epsilon + k\chi) = \dot\varphi_\epsilon - n + \tr{\omega_{\varphi_\epsilon}}{\omega_t} ,\]

\[ (\frac \d{\d t} - \Laplace)\dot\varphi_\epsilon = (n-\kappa)-\dot\varphi_\epsilon + \tr{\omega_{\varphi_\epsilon}}{(\widehat\omega_\infty - \omega_t)} ,\]

\[ (\frac \d{\d t} - \Laplace)(e^t \dot\varphi_\epsilon) = e^t(n-\kappa) -\tr{\omega_{\varphi_\epsilon}}{(\omega_0 - \widehat\omega_\infty)} .\]

Therefore,
\[ (\frac \d{\d t} - \Laplace)(e^t \dot\varphi_\epsilon - \dot\varphi_\epsilon - \varphi_\epsilon - k\chi - \kappa t - e^t(n-\kappa)) = - \tr{\omega_{\varphi_\epsilon}}{\omega_0} < 0 .\]

We conclude that the maximum of $(e^t \dot\varphi_\epsilon - \dot\varphi_\epsilon - \varphi_\epsilon - k\chi - \kappa t - e^t(n-\kappa))$ is decreasing, and so
\[e^t \dot\varphi_\epsilon - \dot\varphi_\epsilon - \varphi_\epsilon - k\chi - \kappa t - e^t(n-\kappa) \leq -k\chi -(n-\kappa) \leq 0 .\]
Thus we have for constants independent of $\epsilon$,
\[ (e^t-1)\dot\varphi_\epsilon \leq C + Ct + Ce^t ,\]
and thus for all $t>t_0>0$,
\[ \dot\varphi_\epsilon \leq Ce^{-t} + Cte^{-t} + C \leq C.\]
Combined with the evolution equation ~\eqref{eq:smooth} and the uniform local existence for the parabolic equation this yields,
\[ \dot\varphi_\epsilon \leq C .\]

Next, for the lower bound on $\varphi_\epsilon$ note that by equation ~\eqref{eq:smooth} we have
\[(\omega_t + \ddbar(\varphi_\epsilon + k\chi_\epsilon))^n = e^{-(n-\kappa)t} \frac{e^{\dot\varphi_\epsilon + \varphi_\epsilon + k\chi}}{(\|S\|^2 + \epsilon^2)^{(1-\beta)}} \Omega\]
where the right hand side is uniformly bounded in $L^p(M,\Omega)$ for $1 <p < \frac{1}{1-\beta}$ since $\varphi_\epsilon$, $\dot\varphi_\epsilon$, and $k\chi$ are all bounded above and $\frac{1}{(\|S\|_h^2 + \epsilon^2)^{(1-\beta)}}$ is uniformly bounded in $L^p(M,\Omega)$ as long as $p(1-\beta)<1$.

Thus from \cite{DP,EyGuZe08} we have
\[ osc_M(\varphi_\epsilon + k\chi)<C\]
for uniform constant which is independent of $\epsilon$. Which implies,
\[ osc_M(\varphi_\epsilon)\leq C \]
for constant independent of $\epsilon$.

To get a uniform bound on $|\varphi_\epsilon|$ we need only establish for every $t \in [0,\infty)$,
\[\varphi_\epsilon(x_{t,\epsilon}) > -C\]
for some $x_{t,\epsilon} \in M$ and some uniform constant independent of $\epsilon$.

First we note that,
\[\int_M{e^{\dot{\varphi_\epsilon} + \varphi_\epsilon + k\chi} \frac{\Omega}{(\|S\|_h^2 + \epsilon^2)^{(1-\beta)}}} = e^{(n-\kappa)t} \int_M{(\omega_{t,\epsilon}+\ddbar\varphi_\epsilon)^n} = e^{(n-\kappa)t} \int_M{\omega_t^n} \geq C^{-1},\]
and by the H\"older inequality,
\begin{align*}
	\int_M{e^{\dot{\varphi_\epsilon} + \varphi_\epsilon + k\chi} \frac{\Omega}{(\|S\|_h^2 + \epsilon^2)^{(1-\beta)}}} & \leq (\int_M{e^{q(\dot\varphi_\epsilon + \varphi_\epsilon + k\chi)}\Omega})^{\frac{1}{q}} (\int_M{\frac{\Omega}{(\|S\|_h^2+\epsilon^2)^{p(1-\beta)}}})^{\frac{1}{p}} \\
	& \leq C (\int_M{e^{q(\dot\varphi_\epsilon + \varphi_\epsilon + k\chi)}\Omega})^{\frac{1}{q}}
\end{align*}
for any fixed $1 < p < \frac{1}{1-\beta}$ and $q$ such that $\frac{1}{p} + \frac{1}{q} = 1$.

Thus for each time $t$, there is some point $x_{t,\epsilon}\in M$ such that,
\[e^{\dot{\varphi_\epsilon}(x_{t,\epsilon}) + \varphi_\epsilon(x_{t,\epsilon}) + k\chi(x_{t,\epsilon})} \geq \frac{C^{-1}}{\int_M{\Omega}}. \]
It follows from the uniform upper bounds on $\dot{\varphi_\epsilon}$ and $\chi$ that
\[e^{\varphi_\epsilon(x_{t,\epsilon})} \geq C^{-1}.\]
And so
\[\varphi_\epsilon(x_{t,\epsilon}) \geq -C.\]
We conclude
\[\varphi_\epsilon \geq -C'.\]
Combined with the uniform upper bounds on $\varphi_\epsilon$ this yields,
\[|\varphi_\epsilon| \leq C\]
for constant independent of $\epsilon$. \end{proof}

We now establish the uniform bound on $\dot\varphi_\epsilon$. The proof is essentially contained in the work of Song-Tian \cite{SoTi11}, with the methods generalized to apply to the present situation.

\begin{prop}
There is uniform constant $C>0$ independent of $\epsilon$ such that on $M\times[0,\infty)$
\[ |\dot\varphi_\epsilon| \leq C. \]
\end{prop}

\begin{proof}
From the previous proposition it suffices to establish a lower bound. We consider the family of Monge-Amp\`ere operators for $r\in[0,\infty)$,
\begin{equation}\label{eq:yau}
(\omega_r + \ddbar \xi_{r,\epsilon})^n = e^{-(n-\kappa)r}e^{\xi_{r,\epsilon}} \frac{\Omega}{(\|S\|^2 + \epsilon^2)^{1-\beta}}.
\end{equation}
That a unique solution to ~\eqref{eq:yau} exists for each $\epsilon>0$ and $r\in[0,\infty)$ is due to Aubin \cite{aubin} and Yau \cite{yau}.

By the maximum principle, if $(\xi_{r,\epsilon} - k\chi)$ achieves a maximum at $x_0\in M$, then since 
\[ \omega_r + \ddbar \xi_{r,\epsilon} = \omega_{r,\epsilon} + \ddbar(\xi_{r,\epsilon} - k\chi) \]
it follows that
\[ \xi_{r,\epsilon}(x_0) \leq \log \frac{e^{(n-\kappa)r}\omega_{r,\epsilon}^n(\|S\|_h^2 + \epsilon^2)^{1-\beta}}{\Omega}(x_0) - k\chi(x_0) < C .\]
Since $k\chi$ is uniformly bounded, we therefore have the uniform upper bound 
\[ \xi_{r,\epsilon} \leq C .\]

Now, since
\[ \int_M{e^{\xi_{r,\epsilon}}\frac{\Omega}{(\|S\|^2 + \epsilon^2)^{1-\beta}}} = e^{(n-\kappa)r}\int_M{\omega_r^n} \geq C^{-1} ,\]
and, again by the H\"older inequality with $p$ and $q$ defined as above,
\begin{align*}
	\int_M{e^{\xi_{r,\epsilon}}\frac{\Omega}{(\|S\|^2 + \epsilon^2)^{1-\beta}}}
	\leq & (\int_M{e^{q \xi_{r,\epsilon}} \Omega})^{\frac{1}{q}} (\int_M{\frac{\Omega}{(\|S\|_h^2 + \epsilon^2)^{p(1-\beta)}}})^{\frac{1}{p}} \\
	\leq & C (\int_M{e^{q \xi_{r,\epsilon}} \Omega})^{\frac{1}{q}},
\end{align*}
we get an \textit{a priori} estimate on the supremum,
\[ -C \leq \sup \xi_{r,\epsilon} \leq C .\]

Since the right hand of ~\eqref{eq:yau} is uniformly bounded in $L^p(M,\Omega)$ for $1<p<\frac{1}{1-\beta}$, by \cite{DP,EyGuZe08} we have uniformly
\[ osc_M(\xi_{r,\epsilon})\leq C ,\]
and so we conclude we have the uniform bound,
\[ |\xi_{r,\epsilon}|\leq C, \]
independent of $\epsilon$.

Now, let $\rho(t)$ be a smooth monotone non-increasing function on $[0,1]$ such that
\begin{displaymath}
	\rho(t) = \left \{
		\begin{array}{lr}
			1, & t\in[0,\frac 1 3] \\
			0, & t\in[\frac 2 3, 1]
		\end{array}
	\right.
\end{displaymath}
which we use to define a function on $M\times[0,\infty]$ by
\begin{equation}
Q(x,t) = \rho(t-m)\xi_{m+1,\epsilon}(x) + (1-\rho(t-m))\xi_{m+2,\epsilon}(x) \text{,  for  } (x,t)\in M\times[m,m+1].
\end{equation}
Then $Q$ is smooth and uniformly bounded above and below.

Now we define
\[ H = \dot\varphi_\epsilon + 2\varphi_\epsilon + 2k\chi - Q .\]
Then
\begin{align*}
	(\frac \d{\d t} - \Laplace)H 
	& = \dot\varphi_\epsilon - \frac{\d Q}{\d t} + \tr{\omega_{\varphi_\epsilon}}{(\widehat\omega_\infty + \omega_t + \ddbar Q)} -(n+\kappa)\\
	& \geq \log \frac{e^{(n-\kappa)t}\omega_{\varphi_\epsilon}^n(\|S\|^2 + \epsilon^2)^{1-\beta}}{\Omega} + \tr{\omega_{\varphi_\epsilon}}{(\widehat\omega_\infty + \omega_t + \ddbar Q)} - C,
\end{align*}
and since for $t\in[m,m+1]$, $\widehat\omega_\infty + \omega_t \geq \omega_{m+1}$ and $\widehat\omega_\infty + \omega_t \geq \omega_{m+2}$,
\begin{align*}
\widehat\omega_\infty + \omega_t + \ddbar Q & \geq \frac 1 2 \min(\widehat\omega_\infty + \omega_t + \ddbar\xi_{m+1,\epsilon},\widehat\omega_\infty + \omega_t + \ddbar\xi_{m+2,\epsilon}) \\
& \geq \frac 1 2 \min(\omega_{m+1} + \ddbar\xi_{m+1,\epsilon},\omega_{m+2} + \ddbar\xi_{m+2,\epsilon}).
\end{align*}
We obtain for $t\in[m,m+1]$,
\begin{align*}
	(\widehat\omega_\infty + \omega_t + \ddbar Q)^n 
	& \geq 2^{-n} \min((\omega_{m+1} + \ddbar\xi_{m+1,\epsilon})^n,(\omega_{m+2} + \ddbar\xi_{m+2,\epsilon})^n) \\
	& \geq C^{-1} e^{-(n-\kappa)(m+2)} \frac{\Omega}{(\|S\|^2 + \epsilon^2)^{1-\beta}} \\
	& = C^{-1} e^{-(n-\kappa)(m+2-t)}e^{-(n-\kappa)t}\frac{\Omega}{(\|S\|^2 + \epsilon^2)^{1-\beta}} \\
	& \geq C^{-1} e^{-(n-\kappa)t}\frac{\Omega}{(\|S\|^2 + \epsilon^2)^{1-\beta}}.
\end{align*}

Now we set up the maximum principle. Suppose $H$ achieves a minimum at a point  $(x_0,t_0)\in M\times[0,T]$, with $t_0$ in the interval $[m,m+1]$ for some positive integer $m$. Without loss of generality we may assume $t_0>0$. Then
\begin{align*}
	0 
	\geq (\frac \d{\d t} - \Laplace)H 
	& \geq \log \frac {e^{(n-\kappa)t_0}\omega_{\varphi_\epsilon}^n(\|S\|^2 + \epsilon^2)^{1-\beta}}{\Omega} + \tr{\omega_{\varphi_\epsilon}}{(\widehat\omega_\infty + \omega_{t_0} + \ddbar Q)} - C \\
	& \geq \log \frac {e^{(n-\kappa)t_0}\omega_{\varphi_\epsilon}^n(\|S\|^2 + \epsilon^2)^{1-\beta}}{\Omega} + C^{-1}(\frac{e^{-(n-\kappa)t_0}\Omega}{\omega_{\varphi_\epsilon}^n (\|S\|^2 + \epsilon^2)^{1-\beta}})^{\frac 1 n} - C \\
	& \geq C^{-1}(\frac{e^{-(n-\kappa)t_0}\Omega}{\omega_{\varphi_\epsilon}^n (\|S\|^2 + \epsilon^2)^{1-\beta}})^{\frac 1 n} - C \\
	& \geq C^{-1} e^{-H/n} - C.
\end{align*}
where in the last line we used that $\varphi_\epsilon$, $\chi$, and $Q$ are all uniformly bounded.

We conclude
\[ H \geq -C, \]
and therefore
\[ \dot\varphi_\epsilon \geq -C.\]
Thus $\dot\varphi_\epsilon$ is uniformly bounded independent of $\epsilon$.
\end{proof}

\section{Lower Bound on Scalar Curvature}

%Our goal in this section is to show the scalar curvature is uniformly bounded outside the divisor D along solutions to the conical K\"ahler-Ricci flow. To achieve this, we intend to show that the quantity
%\[ R_\epsilon(t) - (1-\beta)\Laplace_{\omega_{\varphi_\epsilon}}\log(\|S\|_h^2 + \epsilon^2) + (1-\beta)tr_{\omega_{\varphi_\epsilon}}R_h \geq - C \]
%for uniform constant $C>0$ which is independent of $\epsilon$. Here
%\[ R_\epsilon (t) = tr_{\omega_{\varphi_\epsilon}}(\Ric(\omega_{\varphi_\epsilon})) \]
%and the Laplacian is with respect to the exolving metric $\omega_{\varphi_\epsilon}$.
%

The following proposition generalizes the well known result that the scalar curvature is bounded below along solutions of the Ricci flow to solutions of the conical K\"ahler-Ricci flow. In fact, it is not difficult to see that it holds for solutions to ~\eqref{eq:CKRF} with arbitrary simple normal crossing divisors.

For convenience, we define the quantity
\begin{equation}\label{eq:ThetaDef}
	\Theta_\epsilon = (1-\beta)(\ddbar\log(\|S\|_h^2 + \epsilon^2) + R_h).
\end{equation}

\begin{prop}
Let $\omega_{\varphi_\epsilon}$ be a solution to the generalized K\"ahler-Ricci flow ~\eqref{eq:gen}. Then the scalar curvature of $\omega_{\varphi_\epsilon}$,
\[ R_\epsilon(t) = - \tr {\omega_{\varphi_\epsilon}} {\ddbar\log\omega_{\varphi_\epsilon}^n}, \]
satisfies the evolution equation
\[(\frac \d{\d t} - \Laplace) (R_\epsilon(t) - \tr{\omega_{\varphi_\epsilon}} {\Theta_\epsilon}) = |\Ric(\omega_{\varphi_\epsilon}) - \Theta_\epsilon|^2 + R_\epsilon(t) - \tr{\omega_{\varphi_\epsilon}}{\Theta_\epsilon}.\]

In particular we have the lower bound,
\[ R_\epsilon(t) - \tr{\omega_{\varphi_\epsilon}}{\Theta_\epsilon} \geq -n + e^{-t}(n + \inf_{M}(R_\epsilon(\cdot,0) - \tr{\omega_{0,\epsilon}}{\Theta_\epsilon(\cdot)}) .\]
\end{prop}

\begin{proof}

Let $\langle \cdot,\cdot\rangle$ denote the Hermitian inner product associated to the metric $\omega_{\varphi_\epsilon}$.

We calculate,
\[
	\begin{aligned}
		\frac \d{\d t}\tr{\omega_{\varphi_\epsilon}}{(\Ric(\omega_{\varphi_\epsilon}) - \Theta_\epsilon)} 
		& = -\langle \frac \d{\d t} \omega_{\varphi_\epsilon}, \Ric(\omega_{\varphi_\epsilon}) - \Theta_\epsilon \rangle + \tr{\omega_{\varphi_\epsilon}}{(\frac \d{\d t}\Ric(\omega_{\varphi_\epsilon}))} \\
		& = -\langle -\Ric(\omega_{\varphi_\epsilon}) - \omega_{\varphi_\epsilon} + \Theta_\epsilon,\Ric(\omega_{\varphi_\epsilon})-\Theta_\epsilon\rangle - \Laplace \tr{\omega_{\varphi_\epsilon}}{(\frac \d{\d t}\omega_{\varphi_\epsilon})}\\
		& = |\Ric(\omega_{\varphi_\epsilon}) - \Theta_\epsilon|^2 + R_\epsilon(t) - \tr{\omega_{\varphi_\epsilon}}{\Theta_\epsilon} + \Laplace (R_\epsilon(t) - \tr{\omega_{\varphi_\epsilon}}{\Theta_\epsilon})
	\end{aligned}
\]
which proves the evolution equation.

Letting $Q = R_\epsilon(t) - \tr{\omega_{\varphi_\epsilon}}{\Theta_\epsilon}$, and using the well known inequality
\[ |Ric(\omega_{\varphi_\epsilon}) - \Theta_\epsilon|^2 \geq \frac 1n (R_\epsilon(t) - \tr{\omega_{\varphi_\epsilon}}{\Theta_\epsilon})^2 = \frac 1n Q^2 ,\]
it immediately follows that
\begin{align*}
	(\frac \d{\d t} - \Laplace) (e^t (Q + n))
	& \geq e^t(Q+n) + e^t \frac{1}{n} Q (Q + n) \\
	& \geq e^t \frac{1}{n} (Q + n)^2 \geq 0 .
\end{align*}
So we conclude that the minimum of $e^t(Q(\cdot,t)+n)$ is increasing in time, and therefore

\[ (R_\epsilon(t) - \tr{\omega_{\varphi_\epsilon}}{\Theta_\epsilon}) \geq -n + e^{-t}(n + \inf_{M} (R_\epsilon(\cdot,0) - \tr{\omega_{0,\epsilon}}{\Theta_\epsilon(\cdot)})) \]
completing the proof. \end{proof}

We make use of the next lemma to prove a uniform lower bound on $(R_\epsilon(t) - \tr{\omega_{\varphi_\epsilon}}{\Theta_\epsilon})$.

\begin{lem}\label{4.2}
There is a uniform constant $C>0$ independent of $\epsilon$ such that
\[ R_\epsilon(\cdot,0) - \tr{\omega_{0,\epsilon}}{\Theta_\epsilon(\cdot)} \geq -C .\]
\end{lem}

An immediate corollary is the following:

\begin{cor}
$(R_\epsilon(t) - \tr{\omega_{\varphi_\epsilon}}{\Theta_\epsilon})$ is uniformly bounded below along the normalized conical K\"ahler-Ricci flow.
\end{cor}

\begin{proof}[Proof of Lemma \ref{4.2}]
First we recall ~\eqref{eq:ThetaDef}, and 
\[ R_\epsilon(\cdot,0) = -\tr{\omega_{0,\epsilon}}{\ddbar\log \omega_{0,\epsilon}^n} .\]

Rewrite,
\[
	\begin{aligned}
		&   \tr{\omega_{0,\epsilon}}{\Theta_\epsilon} + \tr{\omega_{0,\epsilon}}{\ddbar\log \omega_{0,\epsilon}^n} \\
		& = (1-\beta)\tr{\omega_{0,\epsilon}}{\ddbar\log(\|S\|_h^2+\epsilon^2)} + (1-\beta)\tr{\omega_{0,\epsilon}}{R_h} + \tr{\omega_{0,\epsilon}}{\ddbar\log \omega_{0,\epsilon}^n} \\
		& = \Laplace_{\omega_{0,\epsilon}}\big((1-\beta)\log(\|S\|^2_h + \epsilon^2) + \log\omega_{0,\epsilon}^n\big) + (1-\beta)\tr{\omega_{0,\epsilon}}{R_h} \\
		& = \Laplace_{\omega_{0,\epsilon}}\log\frac{\omega_{0,\epsilon}^n(\|S\|_h^2 + \epsilon^2)^{(1-\beta)}}{\Omega} + \tr{\omega_{0,\epsilon}}{\big((1-\beta)R_h + \ddbar\log\Omega \big)}.
	\end{aligned}
\]

Recall by ~\cite{CGP,GP} there is uniform $\gamma>0$ independent of $\epsilon$ such that
\[\omega_{0,\epsilon} \equiv \omega_0 + k\ddbar\chi(\|S\|^2_h + \epsilon^2) \geq \gamma\omega_0 .\]
Thus
\[ \tr{\omega_{0,\epsilon}}{\big(\ddbar\log\Omega + (1-\beta)R_h \big)} = \tr{\omega_{0,\epsilon}}{\widehat\omega_\infty} \leq \frac{1}{\gamma}\tr{\omega_0}{\widehat\omega_\infty} \leq C.\]
Finally, from \cite{CGP} (see Section 4.5) it has been shown
\[ \Laplace_{\omega_{0,\epsilon}} \log \frac {\omega_{0,\epsilon}^n (\|S\|^2_h + \epsilon^2)^{1-\beta}}{\Omega} \leq C \]
for a uniform constant independent of $\epsilon$.

Thus $\tr{\omega_{0,\epsilon}}{\Theta_\epsilon} + \tr{\omega_{0,\epsilon}}{\ddbar\log\omega_{0,\epsilon}^n}$ is bounded from above independently of $\epsilon$. \end{proof}

\section{Reduction to Laplacian Estimate}
Next we wish to establish a uniform upper bound on $(R_\epsilon(t) - \tr{\omega_{\varphi_\epsilon}}{\Theta_\epsilon})$. Let 
\[ u = \varphi_\epsilon + \dot\varphi_\epsilon + k\chi \]
and
\[ \psi = \tr{\omega_{\varphi_\epsilon}}{\widehat\omega_\infty}, \]
which depend on $\epsilon$, but we suppress this dependence. 

We then have the identity
\begin{equation}\label{eq:id}
\Ric(\omega_{\varphi_\epsilon}) - \Theta_\epsilon = -\ddbar u - \widehat\omega_\infty.
\end{equation}
So it suffices to show $\Laplace u + \psi$ is uniformly bounded below, and since 
\[ 0\leq \psi = \tr{\omega_{\varphi_\epsilon}}{\widehat\omega_\infty} \]
it suffices to prove a lower bound on $\Laplace u$.

We know $u$ is uniformly bounded above and below, and we have the following evolution equations for $u$:
\begin{align*}
& (\frac \d{\d t} - \Laplace) u = - \kappa + \psi, \\
& (\frac \d{\d t} - \Laplace) |\nabla u|^2 = |\nabla u|^2 - |\nabla \overline\nabla u|^2 - |\nabla \nabla u|^2 - \Theta_\epsilon(\nabla u,\overline\nabla u) + 2 \mathrm{Re}  \langle \nabla \psi, \nabla u \rangle ,\\
& (\frac \d{\d t} - \Laplace) \Laplace u = \langle \Ric(\omega_{\varphi_\epsilon}),\ddbar u\rangle - \langle \Theta_\epsilon, \ddbar u \rangle + \Laplace u + \Laplace \psi,
\end{align*}
where in local coordinates we are writing,
\[ \Theta_\epsilon(\nabla u, \overline\nabla u) = \Theta_{\epsilon, i \bar j} \nabla^i u \nabla^{\bar j} u \]
for $\nabla$ the Levi-Civita connection of the evolving metric.

\section{A Parabolic Schwarz Lemma}
We first need to estimate $\psi$. We utilize a parabolic Schwarz lemma adapted to the present setting following the original proof of Song-Tian \cite{SoTi07,SoTi12}.

\begin{prop}\label{1}
Let $\omega_{\varphi_\epsilon}$ be a solution to ~\eqref{eq:gen} and let $\psi= \tr{\omega_{\varphi_\epsilon}}{\widehat\omega_\infty}$. Then $\psi$ is uniformly bounded.
\end{prop}

\begin{lem}\label{2}
$\psi$ satisfies the evolution inequality
\[ (\frac \d{\d t} -\Laplace) \psi \leq - \langle \Theta_\epsilon,\widehat\omega_\infty \rangle + \psi - \frac{|\nabla \psi|^2}{\psi} + C \psi^2 \]
where $C$ is an upper bound for the bisectional curvature of $\omega_Z$.
\end{lem}

\begin{proof}[Proof of Lemma \ref{2}]
Since $\widehat\omega_\infty$ is the pullback of the metric $\omega_Z$ via the morphism $\pi: M \rightarrow Z$ we have as in Song-Tian \cite{SoTi07,SoTi12} the estimate,
\[ \Laplace \psi \geq \langle \Ric(\omega_{\varphi_\epsilon}),\widehat\omega_\infty\rangle + \frac{|\nabla \psi|^2}{\psi} - C \psi^2 ,\]
where $C$ is an upper bound for the bisectional curvature of $\omega_Z$.

For the time derivative we have
\[ \frac \d{\d t} \psi = \langle \Ric(\omega_{\varphi_\epsilon}),\widehat\omega_\infty \rangle - \langle \Theta_\epsilon,\widehat\omega_\infty \rangle + \psi ,\]
so that
\[ (\frac \d{\d t} - \Laplace) \psi \leq  - \langle \Theta_\epsilon,\widehat\omega_\infty \rangle + \psi - \frac{|\nabla \psi|^2}{\psi} + C \psi^2 ,\]
as required. \end{proof}

Before we estimate the term $\langle \Theta_\epsilon, \widehat\omega_\infty \rangle$ we establish some useful calculations.

We have,
\[ \ddbar \log (\|S\|^2_h + \epsilon^2) = - \frac{\|S\|_h^2}{(\|S\|^2_h + \epsilon^2)^2} \sqrt{-1} \nabla S\wedge\overline{\nabla S} + \frac{\sqrt{-1} \nabla S\wedge\overline{\nabla S}}{\|S\|^2_h + \epsilon^2} - \frac{\|S\|^2_h}{\|S\|^2_h + \epsilon^2} R_h \]
with $\sqrt{-1} \nabla S\wedge\overline{\nabla S}$ defined as before. It follows that
\begin{equation}\label{eq:theta}
	\Theta_\epsilon = (1-\beta)\frac{\epsilon^2}{\|S\|_h^2 + \epsilon^2}(\frac{\sqrt{-1} \nabla S\wedge\overline{\nabla S}}{\|S\|_h^2 + \epsilon^2} + R_h).
\end{equation}

Now, since 
\[ 0 < \frac{\epsilon^2}{\|S\|_h^2+\epsilon^2} \leq 1 ,\]
and $\frac{\sqrt{-1} \nabla S\wedge\overline{\nabla S}}{\|S\|_h^2 + \epsilon^2}$ is non-negative, using ~\eqref{eq:C0} we have
\[ \Theta_\epsilon \geq - C \widehat\omega_\infty .\]

\begin{proof}[Proof of Proposition \ref{1}]

From the calculation above we obtain
\begin{align*}
-\langle \Theta_\epsilon, \widehat\omega_\infty \rangle 
& \leq C |\widehat\omega_\infty|^2 \\
& \leq C \psi^2.
\end{align*}
Thus we have,
\begin{align*}
(\frac \d{\d t} - \Laplace) \log \psi & \leq - \frac{\langle \Theta_\epsilon, \widehat\omega_\infty \rangle}{\psi} + 1 + C\psi \\
& \leq 1 + C \psi,
\end{align*}
and so for $A$ sufficiently large constant independent of $\epsilon$ we have,
\[ (\frac \d{\d t} - \Laplace)(\log \psi - A u) \leq C - \psi .\]
If $\log \psi - A u$ achieves a maximum at $(x_0,t_0)$ then the maximum principle gives the uniform upper bound $\psi(x_0,t_0) \leq C$. Thus using the uniform bounds on $u$:
\[ \log \psi - A u \leq \log \psi(x_0,t_0) - A u(x_0,t_0) \leq C .\]
So we conclude that there is a uniform constant $C>0$ such that
\[ \psi \leq C, \]
completing the proof. \end{proof}

\section{The Gradient Estimate}

We use the method of Cheng-Yau \cite{ChYa75} (see also \cite{SesTi,SoTi11,Zh}) and consider $\Psi = \frac{|\nabla u|^2}{(B-u)}$ where $B>0$ is a constant taken large enough that the denominator is positive, and bounded above and away from zero. Note that because our bounds do not depend on $\epsilon$, $B$ can be chosen independent of $\epsilon$. Now we calculate,
\begin{multline}\label{eq:Psi}
	(\frac \d{\d t} - \Laplace)\Psi = \frac{1}{(B-u)} (|\nabla u|^2 - |\nabla \overline\nabla u|^2 - |\nabla \nabla u|^2 - \Theta_\epsilon(\nabla u,\overline\nabla u) +  2 \mathrm{Re} \langle \nabla\psi,\nabla u\rangle) \\
	+ \frac{|\nabla u|^2}{(B-u)}(-\kappa + \psi) - \frac{2}{(B-u)^2} \mathrm{Re}  \langle \nabla |\nabla u|^2,\nabla u\rangle - 2\frac{|\nabla u|^4}{(B-u)^3}. 
\end{multline}
By ~\eqref{eq:theta} we deduce
\begin{align*}
	-\Theta_\epsilon(\nabla u, \overline\nabla u) \leq & C \widehat\omega_\infty(\nabla u,\overline\nabla u) \\
	\leq &  C (\tr{\omega_{\varphi_\epsilon}}{\widehat\omega_\infty}) \cdot \omega_{\varphi_\epsilon}(\nabla u,\overline\nabla u) = C \psi |\nabla u|^2\\
	\leq & C |\nabla u|^2,
\end{align*}
and for small positive constant $\delta$,
\[ |\langle \nabla \psi, \nabla u \rangle | \leq \delta |\nabla \psi|^2 + C |\nabla u|^2 .\]
Noting that $\nabla (\frac{|\nabla u|^2}{B-u}) = \frac{\nabla |\nabla u|^2}{B-u} + \frac{|\nabla u|^2 \nabla u}{(B-u)^2}$, it follows
\begin{align*}
(\frac \d{\d t} - \Laplace)\Psi \leq & C |\nabla u|^2 - \frac 1{B-u}(|\nabla \overline\nabla u|^2 + |\nabla \nabla u|^2) + \delta |\nabla \psi|^2 \\
& - (2-\delta)\mathrm{Re} \langle \nabla \Psi,\frac{\nabla u}{B-u} \rangle - \delta \frac{|\nabla u|^4}{(B-u)^3} - \delta \mathrm{Re} \langle \frac{\nabla |\nabla u|^2}{B-u},\frac{\nabla u}{B-u}\rangle.
\end{align*}
Next we use the estimate from \cite{SesTi,Zh},
\begin{align*}
\delta |\langle \frac{\nabla|\nabla u|^2}{B-u}, \frac{\nabla u}{B-u} \rangle| & \leq \delta \frac{|\nabla u|^2(|\nabla \overline\nabla u| + |\nabla \nabla u|)}{(B-u)^2} \\
& \leq \sqrt{2} \delta \frac{|\nabla u|^2(|\nabla \overline\nabla u|^2 + |\nabla \nabla u|^2)^{\frac 1 2}}{(B-u)^{\frac 3 2}(B-u)^{\frac 1 2}} \\
& \leq \frac \delta 2 \frac{|\nabla u|^4}{(B-u)^3} + \delta \frac{(|\nabla \overline\nabla u|^2 + |\nabla \nabla u|^2)}{(B-u)}.
\end{align*}
Thus for $\delta$ sufficiently small we obtain
\[ (\frac \d{\d t} - \Laplace)\Psi \leq C |\nabla u|^2 + \delta |\nabla \psi|^2 - (2-\delta)\mathrm{Re} \langle \nabla \Psi,\frac{\nabla u}{B-u} \rangle - \frac \delta2 \frac {|\nabla u|^4}{(B-u)^3}. \]
Now since $\psi$ is uniformly bounded above independent of $\epsilon$, for $\delta$ sufficiently small
\begin{equation}\label{eq:psi}
	(\frac d {dt} - \Laplace)\psi \leq -\langle \Theta_\epsilon,\widehat\omega_\infty \rangle + \psi - \frac{|\nabla \psi|^2}\psi + C\psi^2 \leq - 2 \delta |\nabla \psi|^2 + C.
\end{equation}
Noting that in addition
\[ (2-\delta)|\langle \nabla \psi , \frac{\nabla u}{B-u} \rangle| \leq \delta |\nabla \psi|^2 + C |\nabla u|^2 ,\]
we arrive at,
\[ (\frac \d{\d t} -\Laplace)(\Psi + \psi) \leq C + C|\nabla u|^2 - (2-\delta)\mathrm{Re} \langle \nabla(\Psi + \psi),\frac{\nabla u}{B-u} \rangle - \frac \delta2 \frac {|\nabla u|^4}{(B-u)^3} .\]
Then at a maximal point of $(\Psi + \psi)$ we have
\[ 0 \leq C + C|\nabla u|^2 - \frac \delta2 \frac {|\nabla u|^4}{(B-u)^3} \]
for uniform constants $C,\delta>0$ and with $\frac 1{B-u}$ uniformly bounded from below away from zero. It follows that at the maximum we have
\[ |\nabla u|^2 < C ,\]
for a uniform constant $C>0$.

We readily conclude $(\Psi + \psi)$ is uniformly bounded above, and hence $|\nabla u|^2$ is uniformly bounded above.

\section{Laplacian Estimate}

Let $\Phi = \frac{B - \Laplace u - \psi}{B - u}$. Since we have an upper bound on $\Laplace u + \psi$ the constant $B>0$ can be chosen such that the numerator and denominator are both positive and the denominator is bounded above and away from zero, and again because the bounds do not depend on $\epsilon$, $B$ can be fixed independent of $\epsilon$.

Now, since $\nabla(\frac{B-\Laplace u - \psi}{B-u}) = - \frac{\nabla \Laplace u}{B-u} - \frac{\nabla \psi}{B-u} + \frac{(B-\Laplace u - \psi)\nabla u}{(B-u)^2}$,
\begin{align*}
	(\frac \d{\d t} - \Laplace)\Phi = & \frac{-1}{B-u} (\frac \d{\d t} - \Laplace) (\Laplace u + \psi) + \frac {B-\Laplace u-\psi} {B-u} (\frac \d {\d t} -\Laplace)u \\
	   & + \frac 2 {(B-u)^2} \mathrm{Re}  \langle \nabla \Laplace u, \nabla u \rangle + \frac 2 {(B-u)^2} \mathrm{Re}  \langle \nabla \psi, \nabla u \rangle - 2 \frac {B-\Laplace u - \psi}{(B-u)^3} |\nabla u|^2 \\
	=  & \frac {-1}{B-u} (\langle \Ric(\omega_{\varphi_\epsilon}) - \Theta_\epsilon, \ddbar u + \widehat\omega_\infty \rangle + \Laplace u + \psi) \\
	   & + \frac{B - \Laplace u -\psi}{B-u}(-\kappa + \psi) - 2 \mathrm{Re}  \langle \nabla \Phi,\frac{\nabla u}{B-u}\rangle.
\end{align*}
By ~\eqref{eq:id}, for small constant $0<\delta<1$ and a uniform constant depending on $\delta$,
\[ - \langle \Ric(\omega_{\varphi_\epsilon}) - \Theta_\epsilon, \ddbar u + \widehat\omega_\infty \rangle = |\ddbar u + \widehat\omega_\infty|^2 \leq (1 + \delta)|\nabla \overline\nabla u|^2 + C|\widehat\omega_\infty|^2 ,\]
and, since $|\widehat\omega_\infty|^2 \leq \psi^2 \leq C$, we have
\begin{align*}
	(\frac \d{\d t} - \Laplace)\Phi \leq & \frac{-1}{B-u} (-|\ddbar u + \widehat\omega_\infty|^2 + \Laplace u + \psi)\\
   & + C(B-\Laplace u -\psi) - 2 \mathrm{Re}  \langle \nabla \Phi, \frac{\nabla u}{B-u} \rangle \\
	\leq & \frac{(1+\delta)|\nabla \overline\nabla u|^2}{B-u} + C + C(B-\Laplace u) - 2 \mathrm{Re}  \langle \nabla \Phi, \frac{\nabla u}{B-u} \rangle.
\end{align*}
Now from ~\eqref{eq:Psi} and ~\eqref{eq:psi} (using that $|\nabla u|^2$ is now uniformly bounded) we have
\[ (\frac \d{\d t} - \Laplace)(\Psi + \psi) \leq C - \frac{|\nabla \overline\nabla u|^2}{B-u} - 2 \mathrm{Re}  \langle \nabla(\Psi + \psi),\frac{\nabla u}{B-u} \rangle .\]
Thus
\[ (\frac \d{\d t} - \Laplace)(\Phi + 2\Psi + 2\psi) \leq C + C(B-\Laplace u) - C^{-1} |\nabla \overline\nabla u|^2 - 2 \mathrm{Re}  \langle \nabla (\Phi + 2\Psi + 2\psi), \frac{\nabla u}{B-u} \rangle .\]
Lastly since
\[ |\nabla \overline\nabla u|^2 \geq \frac 1n (\Laplace u)^2 \geq \frac 1n (B-\Laplace u)^2 - \frac{B^2}n ,\]
we conclude that if $(\Phi + 2\Psi + 2\psi)$ achieves a maximum at $(x_0,t_0)\in M \times [0,T]$, then
\[ 0 \leq C + C(B - \Laplace u(x_0,t_0)) - C^{-1}(B-\Laplace u(x_0,t_0))^2 \]
for uniform constants which do not depend on $\epsilon$, which implies there is a uniform constant $C>0$ also independent of $\epsilon$ such that
\[ -\Laplace u \leq C .\]
Since $\psi$ is uniformly bounded we conclude that $\Phi = \frac{B-\Laplace u - \psi}{B-u}$ is uniformly bounded above and therefore
\begin{equation}\label{eq:finale}
	|R_\epsilon(t) - \tr{\omega_{\varphi_\epsilon}}{\Theta_\epsilon}| = |\Laplace u + \psi| \leq C
\end{equation}
for uniform constant independent of $\epsilon$. 

\section{Convergence}

Note that as $\epsilon$ goes to zero we have,
\[ \Theta_\epsilon = (1-\beta)(\ddbar\log(\|S\|^2_h + \epsilon^2) + R_h) \rightarrow 2\pi(1-\beta)[D] \]
where $[D]$ is the current of integration along the divisor $D$ and the convergence is globally on $M$ in the sense of currents, and in $C^\infty_{loc}$ on $M \setminus D$.

Now, letting $\epsilon$ tend to zero there is a subsequence $\epsilon_i$ such that $\varphi_{\epsilon_i}$ converges, in $C^\infty_{loc}$ on $M \setminus D$ and in the sense of currents globally on $M$, to the unique solution, $\varphi$, to the parabolic complex Monge-Amp\`ere equation ~\eqref{eq:CMA} (cf. \cite{Shen}).

Hence, for any compact subset $K \subseteq M \setminus D$ we have as $\epsilon_i$ tends to zero
\[ \omega_{\varphi_{\epsilon_i}} \rightarrow \omega = \overline\omega_t + \ddbar\varphi \]
where the convergence is in $C^\infty(K)$ and $\omega$ is the unique solution to the conical K\"ahler-Ricci flow ~\eqref{eq:CKRF}.

Thus we also have convergence in $C^\infty(K)$ as $\epsilon_i$ goes to zero for
\begin{align*}
	R_{\epsilon_i}(t) 
	= & \tr{\omega_{\varphi_{\epsilon_i}}}{\Ric(\omega_{\varphi_{\epsilon_i}})} \\
	= &	-\tr {\omega_{\varphi_{\epsilon_i}}} {\ddbar\log\omega_{\varphi_{\epsilon_i}^n}} \\
	\rightarrow & - \tr {\omega}{\ddbar\log\omega^n} \\
	= & R(t).
\end{align*}

Finally, since
\[ \Theta_\epsilon \rightarrow 0 \]
in $C^\infty(K)$ as $\epsilon$ goes to zero, from the uniform bound ~\eqref{eq:finale}, we conclude that we have the uniform bound on $K$,
\[ |R(t)| \leq C .\]
But since the constant is independent of $K\subseteq M\setminus D$ we conclude the uniform bound holds on all of $M\setminus D$. Thus the scalar curvature is bounded away from $D$ which completes the proof of the main theorem.

\bibliographystyle{plain}
\bibliography{Bibliography}

\end{document}